\documentclass[a4paper]{amsart}
\usepackage{amssymb}
\usepackage{stmaryrd}
\usepackage{fancyhdr}
\usepackage[all]{xy}
\usepackage{epsf}
\usepackage{enumerate}
\usepackage{combelow}
\newtheorem{theorem}{Theorem}[section]
\newtheorem{mtheorem}{Theorem}

\newtheorem{proposition}[theorem]{Proposition}
\newtheorem{lemma}[theorem]{Lemma}
\newtheorem{corollary}[theorem]{Corollary}

\theoremstyle{remark}

\newcommand{\rmap}{\longrightarrow}
\newcommand{\diffto}{\xrightarrow{\raisebox{-0.2 em}[0pt][0pt]{\smash{\ensuremath{\sim}}}}}

\usepackage{amsmath,amscd}

\begin{document}
\title{Deformations of the Lie-Poisson sphere of a compact semisimple Lie algebra}
\author{Ioan M\u{a}rcu\cb{t}}
\address{Depart. of Math., Utrecht University, 3508 TA Utrecht, The Netherlands}
\address{\emph{Current address:} University of Illinois at Urbana-Champaign, Urbana, IL 61801 USA}
\email{marcut@illinois.edu}
\begin{abstract}
A compact semisimple Lie algebra $\mathfrak{g}$ induces a Poisson structure $\pi_{\mathbb{S}}$ on the unit sphere
$\mathbb{S}(\mathfrak{g}^*)$ in $\mathfrak{g}^*$. We compute the moduli space of Poisson structures on
$\mathbb{S}(\mathfrak{g}^*)$ around $\pi_{\mathbb{S}}$. This is the first explicit computation of a Poisson moduli
space in dimension greater or equal than three around a degenerate (i.e.\ not symplectic) Poisson structure.
\end{abstract}

\maketitle

\setcounter{tocdepth}{1}


\section*{Introduction}

Recall that a \textbf{Poisson structure} on a manifold $M$ is a Lie bracket $\{\cdot,\cdot\}$ on $C^{\infty}(M)$, satisfying the Leibniz rule
\[\{f,gh\}=\{f,g\}h+\{f,h\}g.\]
Equivalently, a Poisson structure is given by bivector field $\pi\in\mathfrak{X}^2(M)$ satisfying $[\pi,\pi]=0$ for the Schouten bracket; the
bivector and the Lie bracket are related by
\[\{f,g\}=\langle\pi|df\wedge dg\rangle.\]
The \textbf{Hamiltonian} vector field of a function $f\in C^{\infty}(M)$ is $X_f:=\{f,\cdot\}\in \mathfrak{X}(M)$. These vector fields span a
singular involutive distribution, which integrates to a partition of $M$ into regularly immersed submanifolds called \textbf{symplectic leaves};
each such leaf $S$ carries canonically a symplectic structure: $\omega_S:=\pi_{|S}^{-1}\in\Omega^2(S)$. A smooth function, constant
along the symplectic leaves is called a \textbf{Casimir} function. We denote the space of Casimirs by $\mathfrak{Casim}(M,\pi)$.\\

Lie theory provides interesting examples of Poisson manifolds. The dual vector space of a Lie algebra $(\mathfrak{g},[\cdot,\cdot])$ carries a
canonical a Poisson structure $\pi_{\mathfrak{g}}$, given by
\[\langle(\pi_{\mathfrak{g}})_{\xi}|X\wedge Y\rangle:=\xi([X,Y]), \ \ \xi\in\mathfrak{g}^*, \ X,Y\in\mathfrak{g}=T^*_{\xi}\mathfrak{g}^*.\]
Assuming that $\mathfrak{g}$ is compact and semisimple, $\mathfrak{g}^*$ carries an $Aut(\mathfrak{g})$-invariant inner product (e.g. induced by
the Killing form). The corresponding unit sphere around the origin, denoted by $\mathbb{S}(\mathfrak{g}^*)$, inherits a Poisson structure
$\pi_{\mathbb{S}}:=\pi_{\mathfrak{g}|\mathbb{S}(\mathfrak{g}^*)}$. We will call $(\mathbb{S}(\mathfrak{g}^*),\pi_{\mathbb{S}})$ the
\textbf{Lie-Poisson sphere} corresponding to $\mathfrak{g}$. Lie algebra automorphisms of $\mathfrak{g}$ restrict to Poisson diffeomorphisms of
$\pi_{\mathbb{S}}$, and the inner automorphisms act trivially on Casimirs. Therefore $Out(\mathfrak{g})$, the group of outer automorphisms of
$\mathfrak{g}$, acts naturally on $\mathfrak{Casim}(\mathbb{S}(\mathfrak{g}^*),\pi_{\mathbb{S}})$.

Our main result describes all Poisson structures on $\mathbb{S}(\mathfrak{g}^*)$ near $\pi_{\mathbb{S}}$.

\begin{mtheorem}\label{Theorem_1}
For the Lie-Poisson sphere $(\mathbb{S}(\mathfrak{g}^*),\pi_{\mathbb{S}})$, corresponding to a compact
semisimple Lie algebra $\mathfrak{g}$, the following hold:
\begin{enumerate}[(a)]
\item There exists a $C^p$-open $\mathcal{W}\subset \mathfrak{X}^2(\mathbb{S}(\mathfrak{g}^*))$ around $\pi_{\mathbb{S}}$, such that every Poisson structure in $\mathcal{W}$ is isomorphic to one
of the form $f\pi_{\mathbb{S}}$, where $f$ is a positive Casimir, by a diffeomorphism isotopic to the identity.
\item For two positive Casimirs $f$ and $g$, the Poisson manifolds $(\mathbb{S}(\mathfrak{g}^*),f\pi_{\mathbb{S}})$ and
$(\mathbb{S}(\mathfrak{g}^*),g\pi_{\mathbb{S}})$ are isomorphic precisely when $f$ and $g$ are related by an
outer automorphism of $\mathfrak{g}$.
\end{enumerate}
\end{mtheorem}

The open $\mathcal{W}$ will be constructed such that it contains all Poisson structures of the form $f\pi_{\mathbb{S}}$, with $f$ a positive
Casimir. Therefore, the map $F\mapsto e^{F}\pi_{\mathbb{S}}$ induces a bijection between the space
\begin{equation*}
\mathfrak{Casim}(\mathbb{S}(\mathfrak{g}^*),\pi_{\mathbb{S}})/Out(\mathfrak{g})
\end{equation*}
and an open around $\pi_{\mathbb{S}}$ in the Poisson moduli space of $\mathbb{S}(\mathfrak{g}^*)$. Using classical invariant theory, we show
that this space is isomorphic to
\[C^{\infty}(\overline{B})/Out(\mathfrak{g}),\]
where $B\subset \mathbb{R}^{l-1}$ is a bounded open which is invariant under a linear action of $Out(\mathfrak{g})$ on $\mathbb{R}^{l-1}$ and
$l=\mathrm{rank}(\mathfrak{g})$.\\

The space of Casimirs is the 0-th group of the \textbf{Poisson cohomology} of $(M,\pi)$, computed by the complex of multivector fields on $M$
with differential $d_{\pi}:=[\pi,\cdot]$
\[(\mathfrak{X}^{\bullet}(M),d_{\pi}), \ d_{\pi}(W)=[\pi,W].\]
The first cohomology group $H^1_{\pi}(M)$ represents the infinitesimal automorphisms of $\pi$ modulo those coming
from Hamiltonian vector fields and $H^2_{\pi}(M)$ has the heuristic interpretation of being the ``tangent space'' to
the Poisson moduli space at $\pi$. As our result suggests, for the Lie-Poisson sphere we have an isomorphism between
$\mathfrak{Casim}(\mathbb{S}(\mathfrak{g}^*),\pi_{\mathbb{S}})\cong
H^2_{\pi_{\mathbb{S}}}(\mathbb{S}(\mathfrak{g}^*))$ given by multiplication with $[\pi_{\mathbb{S}}]$ (for a proof see \cite{teza}).\\

There are only few descriptions, in the literature, of opens in the Poisson moduli space of a compact manifold, and we recall below two such
results.

For a compact symplectic manifold $(M,\omega)$, every Poisson structure $C^0$-close to $\omega^{-1}$ is
symplectic as well. The Moser argument shows that two symplectic structures in the same cohomology class, and
which are close enough to $\omega$, are symplectomorphic by a diffeomorphism isotopic to the identity. This implies
that the map $\pi\mapsto [\pi^{-1}]\in H^2(M)$ induces a bijection between an open in the space of all Poisson
structures modulo diffeomorphisms isotopic to the identity and an open in $H^2(M)$. Also the heuristic prognosis
holds, since $H^2(M)\cong H^2_{\omega^{-1}}(M)$. In general it is difficult to say more, that is, to determine
whether two symplectic structures, different in cohomology, are symplectomorphic. In Corollary
\ref{Corollary_moduli_on_coadjoint} we achieve this for the maximal coadjoint orbits of a compact semisimple Lie
algebra.

Radko obtains in \cite{Radko} a description of the moduli space of topologically stable bivectors on a compact oriented
surface $\Sigma$. These are bivectors $\pi\in \mathfrak{X}^2(\Sigma)$ that intersect the zero section of $\Lambda^2
T\Sigma$ transversally, and therefore form a dense $C^1$-open in $\mathfrak{X}^2(\Sigma)$. The moduli space
decomposes as a union of finite dimensional manifolds (of different dimensions), and its tangent space at
$\pi$ is precisely $H^2_{\pi}(\Sigma)$. Since $\Sigma$ is two-dimensional, every bivector $\pi\in \mathfrak{X}^2(\Sigma)$ is Poisson.\\

The main difficulty when studying deformations of Poisson structures on compact manifolds (in contrast, for example,
to complex structures) is that the Poisson complex fails to be elliptic, unless the structure is symplectic. Therefore, in
general, $H^2_{\pi}(M)$ and the Poisson moduli space are infinite dimensional. This is also the case for the
Lie-Poisson spheres, except for $\mathfrak{g}=\mathfrak{su}(2)$. The Lie algebra $\mathfrak{su}(2)$ is special also
because it is the only one for which the Lie-Poisson sphere is symplectic (thus the result follows from Moser's
theorem). Moreover, it is only the one for which the Lie-Poisson sphere is an integrable Poisson manifold (in the sense
of
\cite{CrFe2}).\\

\noindent \textbf{The outline of the paper.} In the first section we prove part (a) of the theorem. This is done by
realizing Poisson structures of the form $f\pi_{\mathbb{S}}$ as Poisson submanifolds of
$(\mathfrak{g}^*,\pi_{\mathfrak{g}})$, and then using the rigidity theorem around compact Poisson submanifolds
from \cite{Marcut}. In section 2 we discuss some standard results from Lie theory and a result from \cite{Papadima},
stating that all diffeomorphisms of a maximal coadjoint orbit are represented cohomologically by Lie group
automorphisms. In section 3 we conclude the proof of the theorem. Using that the regular part of a Poisson structure
$f\pi_{\mathbb{S}}$, for $f$ a positive Casimir, is a trivial foliation with leaves diffeomorphic to a maximal orbit, we
show that the symplectic structure on the leaves determines $f$ up to an outer automorphism of $\mathfrak{g}$.
Section 4 contains a description of the space of Casimirs. In the
last section we work out the case of $\mathfrak{g}=\mathfrak{su}(3)$.\\

\noindent \textbf{Acknowledgments.} I would like to thank Marius Crainic for his very useful suggestions and comments. This research was
supported by the ERC Starting Grant no. 279729.

\section{Proof of part $(a)$ of Theorem \ref{Theorem_1}}

In this section we will assume some familiarity with the theory of Lie algebroids and Lie groupoids. For definitions and
basic properties we recommend \cite{MackenzieGT}.

Throughout the paper, we fix a compact semisimple Lie algebra $(\mathfrak{g},[\cdot,\cdot])$. These assumptions on
$\mathfrak{g}$ are equivalent to compactness of $G$, the 1-connected Lie group of $\mathfrak{g}$. It is well known
that $1$-connectedness of $G$ also implies that $H^2(G)=0$ (see e.g. \cite{DK}). We also fix an inner product on
$\mathfrak{g}$ (hence also on $\mathfrak{g}^*$), which is not only $G$-invariant, but also
$Aut(\mathfrak{g})$-invariant; for example the negative of the Killing form. Notice that a $G$-invariant inner product
is not automatically $Aut(\mathfrak{g})$-invariant. For example, let $\mathfrak{g}$ be isomorphic to the direct
product $\mathfrak{g}\cong \mathfrak{k}\times\mathfrak{k}$, where $\mathfrak{k}$ is simple Lie algebra of
compact type. The $G$-invariant inner products on $\mathfrak{g}$ are of the form $s \mathrm{pr}_1^*(\kappa)+t
\mathrm{pr}_2^*(\kappa)$, for $s,t>0$, where $\kappa$ is the negative of the Killing form on $\mathfrak{k}$, but the
outer automorphism of $\mathfrak{g}$ that switches the two components stabilizes only inner products for which
$s=t$.

A symplectic groupoid integrating the linear Poisson structure $(\mathfrak{g}^*,\pi_{\mathfrak{g}})$ is
\[(T^*G,\omega_{\mathrm{can}})\rightrightarrows \mathfrak{g}^*.\]
As a Lie groupoid, $T^*G$ is isomorphic to the action groupoid $G\ltimes\mathfrak{g}^*\rightrightarrows
\mathfrak{g}^*$, hence all its $s$-fibers are diffeomorphic to $G$. Since $G$ is compact and $H^2(G)=0$, the
Poisson manifold $(\mathfrak{g}^*,\pi_{\mathfrak{g}})$ satisfies the conditions of Theorem 2 from \cite{Marcut},
and we state below conclusion (a) of this result.

\begin{corollary}[of Theorem 2 in \cite{Marcut}]\label{Corollary_A}
Let $S\subset \mathfrak{g}^*$ be a compact Poisson submanifold. There exists an integer $p\geq 0$ and there exist
(arbitrarily small) open neighborhoods $U\subset \mathfrak{g}^*$ of $S$, such that for every open set $O$ satisfying
$S\subset O\subset \overline{O}\subset U$, there exist
\begin{itemize}
\item an open neighborhood $\mathcal{V}_{O}\subset \mathfrak{X}^2(U)$ of $\pi_{|U}$ in the compact-open $C^p$-topology,
\item a function $\widetilde{\pi}\mapsto \psi_{\widetilde{\pi}}$, which associates to a Poisson structure $\widetilde{\pi}\in
\mathcal{V}_{O}$ an embedding $\psi_{\widetilde{\pi}}:\overline{O}\to {\mathfrak{g}^*}$,
\end{itemize}
such that $\psi_{\widetilde{\pi}}$ is a Poisson diffeomorphism between
\[\psi_{\widetilde{\pi}}:(O,\widetilde{\pi}_{|O})\rmap (\psi_{\widetilde{\pi}}(O),\pi_{|\psi_{\widetilde{\pi}}(O)}),\]
and $\psi$ is continuous at $\widetilde{\pi}=\pi$ (with $\psi_{\pi}=\mathrm{Id}_{\overline{O}}$), with respect to the
$C^p$-topology on the space of Poisson structures and the $C^1$-topology on
$C^{\infty}(\overline{O},{\mathfrak{g}^*})$.
\end{corollary}
We will apply this result to spheres in $\mathfrak{g}^*$. For $f\in C^{\infty}(\mathbb{S}(\mathfrak{g}^*))$, with
$f>0$, consider the following embedded sphere $S_f$ in $\mathfrak{g}^*\backslash\{0\}$
\[S_f:=\left\{\frac{1}{f(\xi)}\xi \ |\  \xi\in \mathbb{S}(\mathfrak{g}^*)\right\},\]
and denote by $\varphi_f: \mathbb{S}(\mathfrak{g}^*)\to S_f$, $\varphi_f(\xi):=\xi/f(\xi)$ the map parameterizing $S_f$. The spheres of type
$S_f$ form a $C^1$-open in the space of all (unparameterized) spheres, namely every sphere $S\subset \mathfrak{g}^*$, for which $0\notin S$ and
the map \[\mathrm{pr}:S\rmap \mathbb{S}(\mathfrak{g}^*),\ \  \xi\mapsto \frac{1}{|\xi|}\xi\] is a diffeomorphism, is of the form $S_f$ for some
positive function $f$ on $\mathbb{S}(\mathfrak{g}^*)$.

\begin{lemma}\label{Lemma_B}
The sphere $S_f$ is a Poisson submanifold if and only $f$ is a Casimir. In this case, the following map is a Poisson diffeomorphism
\[\widetilde{\varphi}_f:(\mathbb{S}(\mathfrak{g}^*)\times\mathbb{R}_{+}, tf\pi_{\mathbb{S}}) \rmap (\mathfrak{g}^*\backslash\{0\},\pi_{\mathfrak{g}}), \ (\xi,t)\mapsto \frac{1}{tf(\xi)}\xi.\]
\end{lemma}
\begin{proof}
Compact Poisson submanifolds of $\mathfrak{g}^*$ are the same as $G$-invariant submanifolds, and Casimirs of $\pi_{\mathbb{S}}$ are the same as
$G$-invariant functions on $\mathbb{S}(\mathfrak{g}^*)$. This implies the first part. For the second part, it suffices to check that
$\widetilde{\varphi}_f^*$ preserves the Lie bracket of $X,Y\in \mathfrak{g}\subset C^{\infty}(\mathfrak{g}^*)$. Using that Casimirs go inside
the bracket and that $\mathbb{S}(\mathfrak{g}^*)$ is a Poisson submanifold, this is straightforward
\begin{align*}
\widetilde{\varphi}_{f}^*(\{X,Y\})=&\frac{1}{tf}\{X,Y\}_{|\mathbb{S}(\mathfrak{g}^*)}=tf\{\frac{1}{tf}X_{|\mathbb{S}(\mathfrak{g}^*)},\frac{1}{tf}Y_{|\mathbb{S}(\mathfrak{g}^*)}\}_{|\mathbb{S}(\mathfrak{g}^*)}=\\
&=tf\{\widetilde{\varphi}_f^*(X),\widetilde{\varphi}_f^*(Y)\}_{|\mathbb{S}(\mathfrak{g}^*)}.\qedhere
\end{align*}
\end{proof}

We are now ready to prove the first part of Theorem \ref{Theorem_1}.
\begin{proof}[Proof of part (a) of Theorem \ref{Theorem_1}]
For every Casimir $f>0$, we construct a $C^p$-open $\mathcal{W}_{f}\subset \mathfrak{X}^2(\mathbb{S}(\mathfrak{g}^*))$ containing
$f\pi_{\mathbb{S}}$, such that every Poisson structure in $\mathcal{W}_f$ is isomorphic to one of the form $g\pi_{\mathbb{S}}$, for $g>0$ a
Casimir, by a diffeomorphism isotopic to the identity. Then $\mathcal{W}:=\cup_{f}\mathcal{W}_f$ satisfies the conclusion.

We will apply Corollary \ref{Corollary_A} to the sphere $S_f$. Let $S_f\subset O\subset U$ be opens as in the corollary,
with $0\notin U$. Denote by $\mathcal{U}_f$ the set of functions $\chi\in C^{\infty}(\overline{O},\mathfrak{g}^*)$
satisfying $0\notin \chi(S_f)$, and for which the map
\[\mathrm{pr}\circ \chi \circ \varphi_{f}:\mathbb{S}(\mathfrak{g}^*)\rmap \mathbb{S}(\mathfrak{g}^*)\]
is a diffeomorphism isotopic to the identity. The first condition is $C^0$-open and the second is $C^1$-open. For the
inclusion $\textrm{Id}_{\overline{O}}$ of $\overline{O}$ in $\mathfrak{g}^*$, we have that $\mathrm{pr}\circ
\textrm{Id}_{\overline{O}}\circ \varphi_f=\textrm{Id}$, thus $\mathcal{U}_f$ is $C^1$-neighborhood of
$\textrm{Id}_{\overline{O}}$. By continuity of $\psi$, there exists a $C^p$-neighborhood $\mathcal{V}_f\subset
\mathfrak{X}^2(U)$ of $\pi_{\mathfrak{g}|U}$, such that $\psi_{\widetilde{\pi}}\in \mathcal{U}_f$, for every
Poisson structure $\widetilde{\pi}$ in $\mathcal{V}_f$. We define the $C^p$-open $\mathcal{W}_f$ as follows
\[\mathcal{W}_f:=\{W\in\mathfrak{X}^2(\mathbb{S}(\mathfrak{g}^*)) | \widetilde{\varphi}_{f,*}(t W)_{|U}\in \mathcal{V}_f\}.\]
By Lemma \ref{Lemma_B}, we have that $\widetilde{\varphi}_{f,*}(t
f\pi_{\mathbb{S}})_{|U}=\pi_{\mathfrak{g}|U}$, thus $f\pi_{\mathbb{S}}\in \mathcal{W}_f$. Let $\overline{\pi}$ be
a Poisson structure in $\mathcal{W}_f$. Then for $\widetilde{\pi}:=\widetilde{\varphi}_{f,*}(t\overline{\pi})_{|U}\in
\mathcal{V}_f$, we have that $\psi_{\widetilde{\pi}}\in\mathcal{U}_f$ is a Poisson map between
\[\psi_{\widetilde{\pi}}:(O,\widetilde{\pi}_{|O})\rmap (U,\pi_{\mathfrak{g}|U}).\]
By the discussion before Lemma \ref{Lemma_B}, the condition that $\mathrm{pr}\circ
\psi_{\widetilde{\pi}}\circ\varphi_f$ is a diffeomorphism, implies that $\psi_{\widetilde{\pi}}(S_f)=S_g$, for some
$g>0$. Since $(\mathbb{S}(\mathfrak{g}^*)\times \{1\},\overline{\pi})$ is a Poisson submanifold of
$(\mathbb{S}(\mathfrak{g}^*)\times \mathbb{R}_{+},t\overline{\pi})$, it follows that
$S_f=\widetilde{\varphi}_f(\mathbb{S}(\mathfrak{g}^*)\times \{1\})$ is a Poisson submanifold of
$(O,\widetilde{\pi}_{|O})$, and since $\psi_{\widetilde{\pi}}$ is a Poisson map, we also have that
$S_g=\psi_{\widetilde{\pi}}(S_f)$ is a Poisson submanifold of $(\mathfrak{g}^*,\pi_{\mathfrak{g}})$. So, by Lemma
\ref{Lemma_B}, $g$ is a Casimir and
\[\varphi_g:(\mathbb{S}(\mathfrak{g}^*),g\pi_{\mathbb{S}})\rmap (S_g,\pi_{\mathfrak{g}|S_g})\]
is a Poisson diffeomorphism. Therefore also the map
\[\varphi_g^{-1}\circ\psi_{\widetilde{\pi}}\circ\varphi_f:(\mathbb{S}(\mathfrak{g}^*),\overline{\pi})\rmap (\mathbb{S}(\mathfrak{g}^*),g\pi_{\mathbb{S}}),\]
is a Poisson diffeomorphism. This map is isotopic to the identity, because $\varphi_g^{-1}=\mathrm{pr}_{|S_g}$, and
by construction $\mathrm{pr}\circ \psi_{\widetilde{\pi}} \circ \varphi_{f}$ is isotopic to the identity.
\end{proof}

\section{Some standard Lie theoretical results}\label{Section_Lie_results}
In this section we recall some results on semisimple Lie algebras, which will be used in the proof. Most of these can be found in standard
textbooks like \cite{Dixmier,DK,Knapp}.

\subsection{Automorphisms}

The group $Aut(\mathfrak{g})$, of Lie algebra automorphisms of $\mathfrak{g}$, contains the normal subgroup $Ad(G)$, of inner automorphisms.
Below we recall two descriptions of the group of \textbf{outer automorphisms} $Out(\mathfrak{g}):=Aut(\mathfrak{g})/Ad(G)$.

We fix $T\subset G$ a maximal torus with Lie algebra $\mathfrak{t}\subset \mathfrak{g}$. Let $Aut(\mathfrak{g},\mathfrak{t})$ be the subgroup of
$Aut(\mathfrak{g})$ consisting of elements which send $\mathfrak{t}$ to itself. Since every two maximal tori are conjugated,
$Aut(\mathfrak{g},\mathfrak{t})$ intersects every component of $Aut(\mathfrak{g})$; hence $Out(\mathfrak{g})\cong
Aut(\mathfrak{g},\mathfrak{t})/Ad(N_G(T))$, where $N_G(T)$ is the normalizer of $T$ in $G$.

Denote by $\Phi\subset i\mathfrak{t}^*$ the corresponding root system, and its symmetry group by
\[Aut(\Phi):=\{f\in Gl(i\mathfrak{t}^*):f(\Phi)=\Phi\}.\]
For $\sigma\in Aut(\mathfrak{g},\mathfrak{t})$, we have that $(\sigma_{|\mathfrak{t}})^*\in Aut(\Phi)$. This gives a group homomorphism
\[\tau: Aut(\mathfrak{g},\mathfrak{t})\rmap Aut(\Phi), \ \sigma\mapsto (\sigma^{-1}_{|\mathfrak{t}})^*.\]

Let $W\subset Aut(\Phi)$ be the \textbf{Weyl group} of $\Phi$. Theorem 7.8 \cite{Knapp} gives the following.

\begin{lemma}\label{Lemma_3}
The map $\tau:Aut(\mathfrak{g},\mathfrak{t})\to Aut(\Phi)$ is surjective, and \[\tau^{-1}(W)=Ad(N_G(T))\subset Aut(\mathfrak{g},\mathfrak{t}).\]
Therefore, $\tau$ induces an isomorphism between $Out(\mathfrak{g})\cong Aut(\Phi)/W$.

Moreover, if $\mathfrak{c}$ is an open Weyl chamber, then $Aut(\Phi)=W\rtimes Aut(\Phi,\mathfrak{c})$, where \[Aut(\Phi,\mathfrak{c}):=\{f\in
Aut(\Phi)| f(\mathfrak{c})=\mathfrak{c}\},\] and $Aut(\Phi,\mathfrak{c})$ is isomorphic to the symmetry group of the Dynkin diagram of $\Phi$.
\end{lemma}

The last part of the lemma allows us to compute $Out(\mathfrak{g})$ for all semisimple compact Lie algebras. First, it is enough to consider
simple Lie algebras, since if $\mathfrak{g}$ decomposes into simple components as $n_1\mathfrak{s}_1\oplus \ldots \oplus n_k \mathfrak{s}_k$,
then \[Out(\mathfrak{g})\cong S_{n_1}\ltimes Out(\mathfrak{s}_1)^{n_1}\times\ldots \times S_{n_k}\ltimes Out(\mathfrak{s}_k)^{n_k}.\] Further,
for the simple Lie algebras, a glimpse at their Dynkin diagram reveals that the only ones with nontrivial outer automorphism group are:
$A_{n\geq 2}$, $D_{n\geq 5}$, $E_6$ with $Out\cong \mathbb{Z}_2$, and $D_{4}$ with $Out\cong S_{3}$.

\subsection{The coadjoint action and its symplectic orbits}

The adjoint action of $G$ on $\mathfrak{g}$ will be denoted by $Ad_g(X)$, the coadjoint action by $Ad_g^{\dagger}(\xi):=\xi \circ Ad_{g^{-1}}$.

The symplectic leaves of the Poisson manifold $(\mathfrak{g}^*,\pi_{\mathfrak{g}})$ are the coadjoint orbits. For $\xi\in\mathfrak{g}^*$, denote
by $(O_{\xi},\Omega_{\xi})$ the symplectic leaf through $\xi$, by $G_{\xi}\subset G$ the stabilizer of $\xi$ and by $\mathfrak{g}_{\xi}\subset
\mathfrak{g}$ the Lie algebra of $G_{\xi}$. The pullback of $\Omega_{\xi}$ to $G$, via the map $g\mapsto Ad_g^{\dagger}(\xi)$, is
$d\widetilde{\xi}$, where $\widetilde{\xi}\in \Omega^1(G)$ is the left invariant extension of $\xi$.

The adjoint representation is isomorphic to the coadjoint representation; an isomorphism between them is induced by
the Killing form. We restate here some standard results about the adjoint action in terms of the coadjoint (as a
reference see section 3.2 in \cite{DK}). We are interested especially in the set $\mathfrak{g}^{*}_{\mathrm{reg}}$
of \textbf{regular} elements. An element $\xi\in\mathfrak{g}^*$ is regular if and only if it satisfies any of the following
equivalent conditions:
\begin{itemize}
\item $\mathfrak{g}_{\xi}$ is a maximal abelian subalgebra;
\item the leaf $O_{\xi}$ has maximal dimension among all leaves;
\item $G_{\xi}$ is a maximal torus in $G$.
\end{itemize}

We regard $\mathfrak{t}^*$ as a subspace of $\mathfrak{g}^*$, by identifying it with
$\mathfrak{t}^*=\{\xi\in\mathfrak{g}^* | \mathfrak{t}\subset\mathfrak{g}_{\xi}\}\subset\mathfrak{g}^*$.
Consider $\mathfrak{t}^{*}_{\mathrm{reg}}:=\mathfrak{t}^{*}\cap \mathfrak{g}^{*}_{\mathrm{reg}}$, the
regular part of $\mathfrak{t}^*$. Then $\mathfrak{t}^{*}_{\mathrm{reg}}$ is the union of the open Weyl chambers.
If $\mathfrak{c}$ is such a chamber, the global structure of $\mathfrak{g}^{*}_{\mathrm{reg}}$ is described by the
equivariant diffeomorphism (Proposition 3.8.1 \cite{DK})
\begin{equation}\label{EQ_regular_part}
\Psi:G/T\times \mathfrak{c} \diffto \mathfrak{g}^{*}_{\mathrm{reg}}, \ \ \Psi([g],\xi)= Ad_{g}^{\dagger}(\xi).
\end{equation}

\subsection{The maximal coadjoint orbits}

The manifold $G/T$ is called a \textbf{generalized flag manifold}, and it is diffeomorphic to all maximal leaves of the linear Poisson structure
on $\mathfrak{g}^*$. Their cohomology is well understood \cite{Borel}; we recall here:

\begin{lemma}\label{Lemma_1}
For $\xi\in\mathfrak{t}^*$, $d\widetilde{\xi}$ is the pullback of a 2-form $\omega_{\xi}$ on $G/T$. Moreover, the
assignment $\mathfrak{t}^*\ni\xi\mapsto [\omega_{\xi}]\in H^2(G/T)$ is a linear isomorphism.
\end{lemma}
\begin{proof}
Since $d\widetilde{\xi}$ is the pullback of the symplectic structure on the symplectic leaf $O_{\xi}$ via the map $G\to
G/G_{\xi}\cong O_{\xi}$, and since $T\subset G_{\xi}$, the first part follows.

Viewing $G$ as a principal $T$-bundle over $G/T$, the long exact sequence for the homotopy groups gives $\pi_1(T)\cong \pi_2(G/T)$ and
$\pi_1(G/T)=0$; thus, using also the Hurewicz theorem, we obtain that the second Betti number of $G/T$ equals $\mathrm{dim}(\mathfrak{t})$. So
it suffices to show injectivity of the map $\xi\mapsto [\omega_{\xi}]$. 
Let $\xi\in \mathfrak{t}^*$, with $\xi\neq 0$. Then we can find an element $X\in\mathfrak{t}$ such that $\xi(X)\neq 0$ and $\exp(X)=1$. Since
$G$ is simply connected, the loop $\gamma_X(t):=\exp(tX)$ is the boundary of some disc $D_X$; and $D_X$ projects to some sphere $S_X$ in $G/T$.
Using Stokes theorem, we obtain that
\[\int_{S_X}\omega_{\xi}=\int_{D_X}d\widetilde{\xi}=\int_{\gamma_X}\widetilde{\xi}=\int_{0}^1\xi(\frac{d}{ds}(\exp((s-t)X))_{|s=t})dt=\xi(X)\neq 0.\]
This shows that $\omega_{\xi}$ is nontrivial in cohomology, and finishes the proof.
\end{proof}

An element $\sigma\in Aut(\mathfrak{g},\mathfrak{t})$ integrates to a Lie group isomorphism of $G$, denoted by the same symbol, which satisfies
$\sigma(T)=T$. Therefore it induces a diffeomorphism $\overline{\sigma}$ of $G/T$. This diffeomorphism has the following property:

\begin{lemma}\label{Lemma_sigma}
We have that $\overline{\sigma}^{*}(\omega_{\xi})=\omega_{\sigma^*(\xi)}$.
\end{lemma}
\begin{proof}
Using that $l_{\sigma(g)^{-1}}\circ \sigma=\sigma \circ l_{g^{-1}}$, the following computation implies the result:
\[\sigma^*(\widetilde{\xi})(X)=\xi(dl_{\sigma(g)^{-1}}\circ d\sigma(X))=\xi(d\sigma\circ d l_{g^{-1}}(X))=\widetilde{\sigma^*(\xi)}(X), \ \forall  X\in T_gG.\qedhere\]
\end{proof}

Every diffeomorphism of $G/T$ induces an algebra automorphism of $H^{\bullet}(G/T)$, and the possible outcomes are covered by the maps
$\overline{\sigma}$. This follows from Theorem 1.2 in \cite{Papadima}, the conclusion we state below (for a self contained exposition see
\cite{teza}).

\begin{proposition}\label{Proposition_1}
For every diffeomorphism $\varphi:G/T\to G/T$, there exists $\sigma\in  Aut(\mathfrak{g},\mathfrak{t})$, such that $\overline{\sigma}:G/T\to
G/T$ induces the same map on $H^{\bullet}(G/T)$ as $\varphi$, i.e.
\[\varphi^*=\overline{\sigma}^*:H^{\bullet}(G/T)\rmap H^{\bullet}(G/T).\]
\end{proposition}
\begin{proof}
By Lemma \ref{Lemma_3}, we can choose $k:=|Aut(\Phi)|$ elements $\sigma_1,\ldots,\sigma_k\in Aut(\mathfrak{g},\mathfrak{t})$, such that
$Aut(\Phi)=\{(\sigma_{1})_{|\mathfrak{t}^*}^*,\ldots,(\sigma_{k})_{|\mathfrak{t}^*}^*\}$. By Lemma \ref{Lemma_sigma}, we have that
$\overline{\sigma}_i^*(\omega_{\xi})=\omega_{\sigma_i^*(\xi)}$, and since the map $\xi\mapsto [\omega_{\xi}]$ is an isomorphism (Lemma
\ref{Lemma_1}), it follows that $\overline{\sigma}_1^*,\ldots,\overline{\sigma}_k^*$ have different actions on $H^{\bullet}(G/T)$. Now, by
Theorem 1.2 in \cite{Papadima} the group of graded automorphisms of $H^{\bullet}(G/T,\mathbb{Z})$ is isomorphic to the group $Aut(\Phi)$, so it
has $k$ elements, and this implies that the $\sigma_i$'s cover all the possible such automorphisms. This finishes the proof.
\end{proof}

The following consequence won't be used in the proof of Theorem \ref{Theorem_1}.

\begin{corollary}\label{Corollary_moduli_on_coadjoint}
The map $\xi\mapsto \omega_{\xi}$ induces a bijection between $\mathfrak{t}^{*}_\mathrm{reg}/Aut(\Phi)$ and an
open in the moduli space of all symplectic structures on $G/T$.
\end{corollary}
\begin{proof}
First, Proposition \ref{Proposition_1}, Lemma \ref{Lemma_sigma} and Lemma \ref{Lemma_3} imply that for
$\xi_1,\xi_2\in\mathfrak{t}^{*}_{\mathrm{reg}}$, we have that $\omega_{\xi_1}$ and $\omega_{\xi_2}$ are
symplectomorphic, if and only if $\xi_1=f(\xi_2)$ for some $f\in Aut(\Phi)$. This shows that the map $\xi\mapsto
\omega_{\xi}$ induces a bijection
\[\Theta:\mathfrak{t}^{*}_{\mathrm{reg}}/Aut(\Phi)\rmap \mathcal{S}/\mathrm{Diff}(G/T),\]
where $\mathcal{S}$ denotes the space of all symplectic form on $G/T$ which are symplectomorphic to one of the type
$\omega_{\xi}$, for some $\xi\in \mathfrak{t}^{*}_{\mathrm{reg}}$. The Moser argument implies that
$\mathcal{S}$ is $C^0$-open in the space of all symplectic forms.
\end{proof}

\section{Proof of part $(b)$ of Theorem \ref{Theorem_1}}
In Poisson geometric terms, $\mathfrak{g}^{*}_{\mathrm{reg}}$ is described as the regular part of
$(\mathfrak{g}^*,\pi_{\mathfrak{g}})$, i.e.\ the open subset consisting of leaves of maximal dimension. The regular
part of $\mathbb{S}(\mathfrak{g}^*)$ is
$\mathbb{S}(\mathfrak{g}^*)_{\mathrm{reg}}=\mathfrak{g}^{*}_{\mathrm{reg}}\cap\mathbb{S}(\mathfrak{g}^*)$.
Let $\mathfrak{c}\subset \mathfrak{t}^*$ be an open Weyl chamber and denote by
$\mathbb{S}(\mathfrak{c}):=\mathfrak{c}\cap \mathbb{S}(\mathfrak{g}^*)$. From (\ref{EQ_regular_part}) it
follows that $\mathbb{S}(\mathfrak{g}^*)_{\mathrm{reg}}$ is described by the diffeomorphism
\[\Psi: G/T\times \mathbb{S}(\mathfrak{c})\diffto \mathbb{S}(\mathfrak{g}^{*})_{\mathrm{reg}}, \ \ \Psi([g],\xi):=Ad_{g}^{\dagger}(\xi),\]
and the symplectic leaves correspond to the slices $(G/T\times \{\xi\},\omega_{\xi})$, $\xi\in\mathfrak{c}$.

\begin{proof}[Proof of part (b) of Theorem \ref{Theorem_1}]
Let $\phi:(\mathbb{S}(\mathfrak{g}^*),f\pi_{\mathbb{S}})\rmap (\mathbb{S}(\mathfrak{g}^*),g\pi_{\mathbb{S}})$
be a Poisson diffeomorphism, where $f,g$ are positive Casimirs. Now, the symplectic leaves of $f\pi_{\mathbb{S}}$
and $g\pi_{\mathbb{S}}$ are also the coadjoint orbits $O_{\xi}$, for $\xi\in \mathbb{S}(\mathfrak{g}^*)$, but with
symplectic structures $1/f(\xi)\omega_{\xi}$, respectively $1/g(\xi)\omega_{\xi}$. In particular, they have the same
regular part $\mathbb{S}(\mathfrak{g}^*)_{\mathrm{reg}}$. So, after conjugating with $\Psi$, the Poisson
diffeomorphism on the regular parts takes the form
\[\Psi^{-1}\circ \phi\circ \Psi: (G/T\times \mathbb{S}(\mathfrak{c}), \Psi^*(f\pi_{\mathbb{S}}))\diffto (G/T\times \mathbb{S}(\mathfrak{c}), \Psi^*(g\pi_{\mathbb{S}})),\]
\[(x,\xi)\mapsto (\phi_{\xi}(x),\theta(\xi)), \]
for a diffeomorphism $\theta:\mathbb{S}(\mathfrak{c})\diffto \mathbb{S}(\mathfrak{c})$ and a symplectomorphism
\[\phi_{\xi}:(G/T,\omega_{\xi/f(\xi)})\diffto (G/T,\omega_{\theta(\xi)/g(\theta(\xi))}).\]
Since $\mathbb{S}(\mathfrak{c})$ is connected it follows that the maps $\phi_{\xi}$ for $\xi\in
\mathbb{S}(\mathfrak{c})$ are isotopic to each other, so they induces the same map on $H^2(G/T)$ and, by
Proposition \ref{Proposition_1}, this map is also induced by a diffeomorphism $\overline{\sigma}$ corresponding to
some $\sigma\in Aut(\mathfrak{g},\mathfrak{t})$. Lemma \ref{Lemma_sigma} implies the following equality in
$H^2(G/T)$ for all $\xi\in \mathbb{S}(\mathfrak{c})$:
\[[\omega_{\xi/f(\xi)}]=[\phi_{\xi}^*(\omega_{\theta(\xi)/g(\theta(\xi))})]=[\overline{\sigma}^*(\omega_{\theta(\xi)/g(\theta(\xi))})]=[\omega_{\sigma^*(\theta(\xi)/g(\theta(\xi)))}].\]
Using Lemma \ref{Lemma_1} we obtain that $\xi/f(\xi)=\sigma^*(\theta(\xi))/g(\theta(\xi))$. Since $\sigma^*$
preserves the norm, we get that $f(\xi)=g(\theta(\xi))$. This shows that $\xi=\sigma^*(\theta(\xi))$, so $\sigma^*$
preserves $\mathbb{S}(\mathfrak{c})$ and, on this space, $\theta=(\sigma^{-1})^*$. So $f\circ \sigma^*(\xi)=g(\xi)$
for all $\xi\in \mathbb{S}(\mathfrak{c})$. Since the regular leaves are dense and all hit $\mathbb{S}(\mathfrak{c})$,
and since both $f\circ \sigma^*$ and $g$ are Casimirs, it follows that $f\circ \sigma^*=g$.
\end{proof}

\section{The space of Casimirs}\label{Section_the_space_of_Casimirs}

By the main theorem, the map which associates to $F\in
\mathfrak{Casim}(\mathbb{S}(\mathfrak{g}^*),\pi_{\mathbb{S}})$ the Poisson structure $e^{F}\pi_{\mathbb{S}}$
on $\mathbb{S}(\mathfrak{g}^*)$ induces a parametrization of an open in the Poisson moduli space of
$\mathbb{S}(\mathfrak{g}^*)$ around $\pi_{\mathbb{S}}$ by the space
\[\mathfrak{Casim}(\mathbb{S}(\mathfrak{g}^*),\pi_{\mathbb{S}})/Out(\mathfrak{g}).\]
In this section we describe this space using classical invariant theory.

Let $P[\mathfrak{g}^*]$ and $P[\mathfrak{t}^*]$ denote the algebras of polynomials on $\mathfrak{g}^*$ and
$\mathfrak{t}^*$ respectively. A classical result (see e.g. Theorem 7.3.5 \cite{Dixmier}) states that the restriction
map $P[\mathfrak{g}^*]\to P[\mathfrak{t}^*]$ induces an isomorphism between the algebras of invariants
\begin{equation}\label{EQ_Invariant_Polynomials}
P[\mathfrak{g}^*]^G\cong P[\mathfrak{t}^*]^W.
\end{equation}
A theorem of Schwarz \cite{Schwarz} extends this result to the smooth setting
\begin{equation}\label{EQ_Invariant_Functions}
C^{\infty}(\mathfrak{g}^*)^G\cong C^{\infty}(\mathfrak{t}^*)^W.
\end{equation}
To explain this, first recall that the algebra $P[\mathfrak{g}^*]^G$ is generated by $l:=\mathrm{dim}(\mathfrak{t})$
algebraically independent homogeneous polynomials $p_1,\ldots,p_l$ (Theorem 7.3.8 \cite{Dixmier}). Hence, by
(\ref{EQ_Invariant_Polynomials}), $P[\mathfrak{t}^*]^W$ is generated by
$q_1:=p_{1|\mathfrak{t}^*},\ldots,q_l:=p_{l|\mathfrak{t}^*}$. Consider the maps
\[p=(p_1,\ldots,p_l):\mathfrak{g}^*\rmap \mathbb{R}^{l},\  \textrm{and}\ q=(q_1,\ldots,q_l):\mathfrak{t}^*\rmap \mathbb{R}^{l},\]
and denote by $\Delta:=p(\mathfrak{g}^*)$. Since the inclusion $\mathfrak{t}^*\subset \mathfrak{g}^*$ induces a
bijection between the $W$-orbits and the $G$-orbits, it follows that $q(\mathfrak{t}^*)=\Delta$. The theorem of
Schwarz \cite{Schwarz} applied to the action of $G$ on $\mathfrak{g}^*$ and to the action of $W$ on
$\mathfrak{t}^*$, shows that the pullbacks by $p$ and $q$ give isomorphisms between (hence we obtain
(\ref{EQ_Invariant_Functions}))
\begin{equation}\label{EQ_Invariant_Functions2}
C^{\infty}(\mathfrak{g}^*)^G\cong C^{\infty}(\Delta), \ \ C^{\infty}(\mathfrak{t}^*)^W\cong C^{\infty}(\Delta).
\end{equation}
Schwarz's result asserts that $p$, respectively $q$, induce homeomorphisms between the orbit spaces and $\Delta$
\begin{equation}\label{EQ_Leaf_space_homeomorphisms}
\mathfrak{g}^*/G\cong \Delta\cong \mathfrak{t}^*/W.
\end{equation}
We can describe the orbit space also using an open Weyl chamber $\mathfrak{c}\subset \mathfrak{t}^*$.

\begin{lemma}\label{Lemma_2}
The map $q:\overline{\mathfrak{c}}\to \Delta$ is a homeomorphism and restricts to a diffeomorphism between the
interiors $q:\mathfrak{c}\to \mathrm{int}(\Delta)$.
\end{lemma}
\begin{proof}
It is well known that $\overline{\mathfrak{c}}$ intersects each orbit of $W$ exactly once (see e.g.\ \cite{DK}) and so,
by (\ref{EQ_Leaf_space_homeomorphisms}), the map is a bijection. Since $q:t^*\to \mathbb{R}^l$ is proper, it follows
that also $q_{|\overline{c}}$ is proper, and this implies the first part. We are left to check that $q_{|\mathfrak{c}}$
is an immersion. Let $V\in T_\xi\mathfrak{c}$ be a nonzero vector. Consider $\chi$ a smooth, compactly supported
function on $\mathfrak{c}$ satisfying $d\chi_{\xi}(V)\neq 0$. Since the action gives a homeomorphism $W\times
\mathfrak{c}\cong \mathfrak{t}^{*}_{\mathrm{reg}}$, $\chi$ has a unique $W$-invariant extension to
$\mathfrak{t}^*$, which is defined on $w\mathfrak{c}$ by $\widetilde{\chi}=w^*(\chi)$, and extended by zero on
$\mathfrak{t}^*\backslash \mathfrak{t}^{*}_{\mathrm{reg}}$. Then, by (\ref{EQ_Invariant_Functions2}),
$\widetilde{\chi}$ is of the form $\widetilde{\chi}=h\circ q$, for some $h\in C^{\infty}(\Delta)$. Differentiating in the
direction of $V$, we obtain that $d_{\xi}q(V)\neq 0$, and this finishes the proof.
\end{proof}

The polynomials $p_1,\ldots,p_l$ are not unique; a necessary and sufficient condition for a set of homogeneous
polynomials to be such a generating system is that their image in $I/I^2$ forms a basis, where $I\subset
P[\mathfrak{g}^*]^G$ denotes the ideal of polynomials vanishing at $0$. Since $I^2$ is $Out(\mathfrak{g})$
invariant, it is easy to see that we can choose $p_1,\ldots,p_l$ such that $p_1(\xi)=|\xi|^2$ and the linear span of
$p_2,\ldots, p_l$ is $Out(\mathfrak{g})$ invariant. This choice endows $\mathbb{R}^l$ with a linear action of
$Out(\mathfrak{g})$, for which $p$ is $Aut(\mathfrak{g})$ equivariant. Moreover, the action is trivial on the first
component and $\{0\}\times\mathbb{R}^{l-1}$ is invariant. The isomorphism $Aut(\Phi)/W\cong Out(\mathfrak{g})$
from Lemma \ref{Lemma_3} shows that also $q$ is equivariant with respect to the actions of $Aut(\Phi)$ and
$Out(\mathfrak{g})$. Thus we have isomorphisms between the spaces
\[C^{\infty}(\mathfrak{g}^*)^G/Aut(\mathfrak{g})\cong C^{\infty}(\mathfrak{t}^*)^W/Aut(\Phi)\cong C^{\infty}(\Delta)/Out(\mathfrak{g}).\]

Notice that every Casimir $f$ on $\mathbb{S}(\mathfrak{g}^*)$ can be extended to a $G$-invariant smooth function on $\mathfrak{g}^*$, 
therefore
\[\mathfrak{Casim}(\mathbb{S}(\mathfrak{g}^*),\pi_{\mathbb{S}})\cong C^{\infty}(\mathfrak{g}^*)^G_{|\mathbb{S}(\mathfrak{g}^*)}.\]
Since $p_1(\xi)=|\xi|^2$, it follows that $p(\mathbb{S}(\mathfrak{g}^*))=\left(\{1\}\times
\mathbb{R}^{l-1}\right)\cap \Delta$. Denoting by $p':=(p_2,\ldots,p_l):\mathfrak{g}^*\to \mathbb{R}^{l-1}$ and by
$\Delta':=p'(\mathbb{S}(\mathfrak{g}^*))$, we have that $C^{\infty}(\Delta')=C^{\infty}(\Delta)_{|\{1\}\times
\Delta'}$. Lemma \ref{Lemma_2} implies that $q':=p'_{|\mathbb{S}(\mathfrak{t}^*)}$ is a homeomorphism between
$\mathbb{S}(\overline{\mathfrak{c}})\cong \Delta'$, which restricts to a diffeomorphism between
$\mathbb{S}(\mathfrak{c})\cong \mathrm{int}(\Delta')$. This shows that $\Delta'=\overline{B}$, where $B$ is a
bounded open, diffeomorphic to an open ball.

With these, we have the following description of the Casimirs.

\begin{corollary}
The polynomial map $p':\mathfrak{g}^*\to \mathbb{R}^{l-1}$ is equivariant with respect to the actions of
$Aut(\mathfrak{g})$ and $Out(\mathfrak{g})$, and $q':=p'_{|\mathfrak{t}^*}$ is equivariant with respect to the
actions of $Aut(\Phi)$ and $Out(\mathfrak{g})$. These maps induce isomorphisms between
\[\mathfrak{Casim}(\mathbb{S}(\mathfrak{g}^*),\pi_{\mathbb{S}})/Out(\mathfrak{g})\cong C^{\infty}(\mathbb{S}(\mathfrak{t}^*))^W/Out(\mathfrak{g})\cong C^{\infty}(\Delta')/Out(\mathfrak{g}),\]
and $Out(\mathfrak{g})$-equivariant homeomorphisms between the spaces
\[\mathbb{S}(\mathfrak{g}^*)/G\cong \mathbb{S}(\mathfrak{t}^*)/W\cong \Delta'.\]
\end{corollary}

\section{The case of $\mathfrak{su}(3)$.}

In this section, we describe our result for the Lie algebra $\mathfrak{g}=\mathfrak{su}(3)$, whose 1-connected Lie
group is $G=\mathbf{SU}(3)$. Recall that
\[\mathfrak{su}(3)=\{A\in M_3(\mathbb{C})| A+A^*=0, \ tr(A)=0\}, \]
\[  \mathbf{SU}(3)=\{U\in M_3(\mathbb{C})| UU^*=I, \ det(U)=1\}.\]
We use the invariant inner product given by the negative of the trace form $(A,B):=-tr(AB)$. Let $\mathfrak{t}$ be the
space of diagonal matrices in $\mathfrak{su}(3)$
\[\mathfrak{t}:=\left\{D(ix_1,ix_2,ix_3):=\left(
                                               \begin{array}{ccc}
                                                 ix_1 & 0 & 0 \\
                                                 0 & ix_2 & 0 \\
                                                 0 & 0 & ix_3 \\
                                               \end{array}
                                             \right) | \ x_j\in\mathbb{R},\ \sum_j x_j=0 \right\}.\]
The corresponding maximal torus is
\[T:=\{D(e^{i\theta_1},e^{i\theta_2},e^{i\theta_3})\ |\ \theta_j\in\mathbb{R},\ \prod_j e^{i\theta_j}=1\}.\]
The Weyl group is $W=S_3$. It acts on $\mathfrak{t}$ as follows
\[\sigma D(ix_1,ix_2,ix_3)=D(ix_{\sigma(1)},ix_{\sigma(2)},ix_{\sigma(3)}), \ \sigma\in S_3.\]
The Dynkin diagram of $\mathfrak{su}(3)$ is $A_2$ (a graph with one edge), so its symmetry group is
$\mathbb{Z}_2$. A generator of $Out(\mathfrak{su}(3))$ is complex conjugation
\[\gamma\in Aut(\mathfrak{su}(3),\mathfrak{t}),\ \gamma(A)=\overline{A}.\]
On $\mathfrak{t}$, $\gamma$ acts by multiplication with $-1$.

Under the identification of $\mathfrak{t}\cong\mathfrak{t}^*$ given by the inner product, the invariant polynomials
$P[\mathfrak{t}]^{S_3}$ are generated by the symmetric polynomials
\[q_1(D(ix_1,ix_2,ix_3))=x_1^2+x_2^2+x_3^2, \ \ q_2(D(ix_1,ix_2,ix_3))=\sqrt{6}(x_1^3+x_2^3+x_3^3).\]
Identifying also $\mathfrak{su}(3)\cong\mathfrak{su}^*(3)$, $q_1$ and $q_2$ are the restriction to $\mathfrak{t}$ of
the invariant polynomials $p_1$, $p_2\in P[\mathfrak{su}^*(3)]^{\mathbf{SU}(3)}$ (which generate
$P[\mathfrak{su}^*(3)]^{\mathbf{SU}(3)}$)
\[ p_1(A)=-tr(A^2), \ \ \ p_2(A)=i\sqrt{6}tr(A^3).\]
Clearly $p_2\circ\gamma=-p_2$. The inner product on $\mathfrak{t}$ is
\[(D(ix_1,ix_2,ix_3),D(ix'_1,ix'_2,ix'_3))=x_1x_1'+x_2x_2'+x_3x_3',\]
and we have that $\mathbb{S}(\mathfrak{t}^*)\cong \mathbb{S}(\mathfrak{t})$ is a circle, isometrically
parameterized by
\[A(\theta):= \frac{cos(\theta)}{\sqrt{2}} D\left(i,-i,0\right)+\frac{sin(\theta)}{\sqrt{6}}D\left(i,i,-2i\right), \  \theta\in [0,2\pi].\]
In polar coordinates on $\mathfrak{t}$, the polynomials $q_1$ and $q_2$ become
\[q_1(rA(\theta))=r^2, \ \ q_2(rA(\theta))=r^3 sin(3\theta).\]
This implies that the space $\Delta$ is given by
\[\Delta=\{(r^2, r^3 sin(3\theta)) | r\geq 0, \ \theta\in [0,2\pi] \}=\{(x,y)\in \mathbb{R}^2 | x^3\geq y^2\}.\]
The map $q:=(q_1,q_2):\mathfrak{t}\to \mathbb{R}^2$, restricted to the open Weyl chamber
\[\mathfrak{c}:=\{rA(\theta)| r>0, \ \theta\in (-\pi/6,\pi/6 )\},\]
is a diffeomorphism onto $\mathrm{int}(\Delta)$. The linear action of $\mathbb{Z}_2=Out(\mathfrak{su}(3))$ on
$\mathbb{R}^2$, for which $q$ is equivariant, is multiplication by $-1$ on the second component. Therefore
$q':=q_2$ is a $\mathbb{Z}_2$-equivariant homeomorphism between
\[q':\mathbb{S}(\overline{\mathfrak{c}})=\{A(\theta) | \theta\in [-\pi/6,\pi/6]\}\diffto \Delta':=[-1,1],\]
which restricts to a diffeomorphism between the interiors.

We conclude that the Poisson moduli space of the 7-dimensional sphere $\mathbb{S}(\mathfrak{su}(3)^*)$ is
parameterized around $\pi_{\mathbb{S}}$ by the space
\[C^{\infty}([-1,1])/\mathbb{Z}_2,\]
where $\mathbb{Z}_2$ acts on $C^{\infty}([-1,1])$ by the involution
\[\gamma(f)(x)=f(-x), \ \ f\in C^{\infty}([-1,1]).\]

\bibliographystyle{amsplain}
\def\lllll{}

\vspace*{.3in}
\end{document}